\documentclass[smallextended,numbook,runningheads]{svjour3}

\usepackage{mathptmx}
\usepackage{amsmath,amsfonts,mathtools}
\DeclareMathOperator{\diag}{diag}
\DeclareMathOperator{\tril}{tril}
\DeclareMathOperator{\triu}{triu}
\DeclareMathOperator{\tridiag}{tridiag}
\usepackage{listings}
\usepackage{subdepth} %for equal subscript positioning in commands such as M_{S}^{-1}N_{S}

\journalname{}

\usepackage{pgfplots, pgfplotstable, tikz}
\pgfplotstableset{col sep=comma}
\pgfplotsset{compat=1.16}

\usepackage{hyperref}

\newcommand{\B}[1]{\mbox{\boldmath $#1$}}

\DeclarePairedDelimiter{\norm}{\lVert}{\rVert}

\newcommand{\col}{\mathord{\,:\,}}

\title{Comparison Theorems for Splittings of M-matrices in (block) Hessenberg Form \thanks{The authors are partially supported by INDAM/GNCS and by the project PRA\_2020\_61 of the University of Pisa.}}

\author{Luca Gemignani \and Federico Poloni}

\institute{
  Gemignani, Luca \at 
  Dipartimento di Informatica\\
  Universit\`a di Pisa \\
  \email{luca.gemignani@unipi.it}
  \and
  Poloni, Federico \at
  Dipartimento di Informatica\\
  Universit\`a di Pisa \\
  \email{federico.poloni@unipi.it}
}

\begin{document}
\maketitle

\begin{abstract}
  Some variants of the (block) Gauss--Seidel iteration  for the solution of  linear systems with $M$-matrices in (block) Hessenberg  form   are discussed.
  Comparison results for the asymptotic convergence rate of some regular splittings are derived: in particular, we prove that for a lower-Hessenberg M-matrix $\rho(P_{GS})\geq \rho(P_S)\geq \rho(P_{AGS})$, where $P_{GS}, P_S, P_{AGS}$ are the iteration matrices of the Gauss--Seidel, staircase, and anti-Gauss--Seidel method. This is a result that does not seem to follow from classical comparison results, as these splittings are not directly comparable. It is shown that the  concept of stair partitioning provides  a
  powerful tool  for the design of new variants
  that are suited for parallel computation.  
	
\end{abstract}

\section{Introduction}
 
Solving block Hessenberg  systems is one of the key issues in numerical simulations
of many scientific and engineering problems.
Possibly singular M-matrix linear systems in  block Hessenberg form  are found in finite difference or finite
element methods for partial differential equations, Markov chains, production and growth models in economics, and linear complementarity problems in operational research~\cite{Nelson,Plemmons}. 
Finite difference or finite element discretizations of PDEs usually produce matrices which are banded or block banded (e.g., block tridiagonal or block pentadiagonal) \cite{Saad}.
Discrete-state   models encountered in   several applications such as  modeling
and analysis of communication  and computer networks can be conveniently represented by 
a discrete/continuous time Markov chain \cite{Nelson}.
In many cases   an appropriate numbering of the states yields a chain with  block  upper or lower Hessenberg structures (GI/M/1 and  M/G/1 queues) or, in the intersection, a
block tridiagonal  generator or transition probability matrix,  that is, a quasi-birth-and-death process (QBD).
QBDs  are   also well  suited  for  modeling  various  population  processes   such as cell growth, biochemical reaction kinetics, epidemics, demographic trends,
or queuing systems, amongst others \cite{LR_redbook}.

Active  computational research  in this area is
focused on the development of techniques, methods and
data structures, which minimize  the computational (space and
time) requirements for solving  large   and possibly sparse linear systems.  
One of such techniques is parallelization.
Divide-and-conquer solvers for  $M$-matrix  linear systems in  block  banded or  block Hessenberg form   are described in \cite{DC,GL,Stewart}.
A specialization of these algorithms based on cyclic reduction for
block Toeplitz Hessenberg matrices is  discussed in \cite{BMC}. 
However, due to the communication costs these  schemes   typically  scale well with processor count only for very large matrix block sizes.
Iterative methods  can  provide an  attractive alternative  primarily because they simplify  both  implementation and sparsity treatment. The crux resides in the analysis of their
convergence properties. 

Among classical iterative methods, the Gauss--Seidel method has several interesting  features.
It is a classical result that on a nonsingular M-matrix the Gauss--Seidel method converges faster than the Jacobi method~\cite[Corollary~5.22]{Plemmons}.
The SOR method with the optimal relaxation parameter can be better yet, but, however, choosing an optimal SOR relaxation parameter is  difficult for many problems.
Therefore, the Gauss--Seidel method is very attractive in practice and it is  also used  as preconditioner in combination with other iterative schemes.  A classical example is the
multigrid method for partial differential equations, where using Gauss--Seidel or SOR as a smoother typically yields good convergence properties \cite{DBLP}.
Parallel implementations of the Gauss--Seidel method have been designed for  certain regular problems,for example, the solution of Laplace's equations by finite differences,
by relying upon red-black coloring or more generally multi-coloring  schemes to provide some parallelism \cite{OV}.
In most cases, constructing efficient parallel true Gauss--Seidel algorithms is challenging  and Processor Block (or localized) Gauss--Seidel is often used \cite{SY}.
Here, each processor performs Gauss--Seidel as a subdomain solver for a block Jacobi method. While Processor Block Gauss--Seidel is easy to parallelize,
the overall convergence  can suffer.  In order to improve the parallelism of Gauss--Seidel-type methods while retaining the same convergence rate, in  \cite{S1}  staircase  splittings are introduced by showing that
for consistently ordered matrices \cite{Saad}  the iterative scheme based on such partitionings splits  into  independent computations and  at the same time exhibits the same convergence rate as the
classical Gauss--Seidel iteration.  An extension of this result  for  block tridiagonal matrices appeared in \cite{AM}. 
 The use of  a  Krylov solver  like BCG,  GMRES and BiCGSTAB,   for block tridiagonal systems  complemented with a  stair preconditioner  which
 accounts  for the structure of the coefficient  matrix  is proposed in \cite{STP}.
 
A classical framework to study the convergence speed of iterative methods for linear systems $A\B x=\B b$  is that of \emph{matrix splittings}: one writes $A = M-N$, with $M$ invertible, and considers the iterative method
\begin{equation}\label{iterativeformula}
  \B x^{(\ell+1)}= P\B x^{(\ell)} + M^{-1} \B b, \quad \ell\geq 0,
\end{equation}
where $P=M^{-1} N$ is the \emph{iteration matrix}. Various results exist to compare the spectral radii of the iteration matrices of two splittings $A = M_1 - N_1 = M_2 - N_2$ under certain elementwise inequalities, such as
\begin{equation} \label{compineqgallery}
    N_2 \geq N_1 \geq 0, \quad M_1^{-1} \geq M_2^{-1}, \quad \text{ or } \quad A^{-1}N_2A^{-1}\geq A^{-1}N_1A^{-1}:
\end{equation}
 see for instance~\cite{Plemmons,CsoV,Woz}. 

In this paper, we consider the solution of $M$-matrix linear systems  
 in (block) Hessenberg form, and we show new comparison results between
 matrix splittings that hold for this special structure.
In particular, for a lower  Hessenberg invertible $M$-matrix $A$  we prove the inequalities
\[
\rho(P_{GS})\geq \rho(P_S)\geq \rho(P_{AGS}),
\]
where $\rho(A)$ denotes the spectral radius of $A$ and $P_{GS}, P_S, P_{AGS}$  are  the iteration matrices  of the Gauss--Seidel method, the staircase splitting  method
 and the anti-Gauss--Seidel method,  respectively.  The first inequality  fosters the  use  of stair partitionings  for solving  Hessenberg  linear systems in parallel.  The second inequality says that the
 anti-Gauss--Seidel method  ---also called Reverse Gauss--Seidel in \cite{vargabook} and Backward Gauss--Seidel in \cite{Saad}--- gives the better choice in terms of convergence speedup.
 Comparison results for more general splittings including some generalizations of staircase partitionings  are also obtained.  

The remarkable feature of these results is that they do not seem to arise
from classical comparison results or from elementwise inequalities of the form~\eqref{compineqgallery} between the matrices that define the splittings:
$M_{GS}$ is lower triangular and $M_{AGS}$ is upper triangular, so they are far from being comparable.

Another reason why our results are counterintuitive is that the intuition behind inequalities of the form~\eqref{compineqgallery} and classical comparison theorems (such as Theorem~\ref{classicalcomparison} in the following) suggests that one should put ``as much of the matrix $A$ as possible'' into $M$ to get a smaller radius: so it is surprising that on a lower Hessenberg matrix AGS, in which $M$ only has $2n-1$ nonzeros, works better than GS, in which $M$ has $\frac{n(n+1)}{2}$ nonzeros, and that this property holds irrespective of the magnitude of these nonzero items.

We give an alternative combinatorial proof of the inequality $\rho(P_{GS}) \geq \rho(P_{AGS})$, which adds a new perspective and shows an elementwise inequality that can be used to derive these bounds. Extensions  to deal with possibly singular $M$-matrices and
 block Hessenberg structures are  discussed. Finally, some numerical experiments confirm the results and give a quantitative estimate of the difference between these spectral radii.

\section{Preliminaries}

Let $A \in \mathbb{R}^{n\times n}$ be an invertible M-matrix.

A \emph{regular splitting} of $A$ is any pair $(M,N)$, where $M$ is an invertible  M-matrix, $N\geq 0$, and $A = M-N$.
 In the terminology of Forsythe \cite{For} a 
\emph{linear stationary iterative method}  for solving
the system $A\B x=\B b$  can be described as~\eqref{iterativeformula}.  Let
\[
A=M-N=M'-N'
\]
be two  regular splittings of $A$. We can derive 
two iterative schemes with iteration matrices $P=M^{-1}N $ and $P'={M'}^{-1} N'$.
When $A$ is nonsingular,  it is well known that the scheme \eqref{iterativeformula}  is
convergent if and only if $\rho(P) < 1$, with $\rho(P)$ \emph{the spectral radius} of $P$.  Under convergence \emph{the asymptotic rate of convergence}
is  also given by $\rho(P)$, and, therefore,  it is interesting to compare $\rho(P)$ and $\rho(P')$.  A classical result is the following.

\begin{theorem}[\cite{vargabook}] \label{classicalcomparison}
Let $(M,N)$ and $(M', N')$ be regular splittings of $A$. If $N \leq N'$, then $\rho(M^{-1}N) \leq \rho((M')^{-1}N')$.
\end{theorem}

Another tool  to obtain comparison results for matrix splittings is the exploitation of certain block partitionings of the matrix $A$. 

\begin{lemma} \label{lem:exchange}
Let
\[
M = \begin{bmatrix}
    M_{11} & 0\\
    A_{21} & M_{22}
\end{bmatrix}, \quad
N = \begin{bmatrix}
    N_{11} & -A_{12}\\
    0 & N_{22}
\end{bmatrix}.
\]
and
\[
\hat{M} = \begin{bmatrix}
    M_{11} & A_{12}\\
    0 & M_{22}
\end{bmatrix}, \quad
\hat{N} = \begin{bmatrix}
    N_{11} & 0\\
    -A_{21} & N_{22}
\end{bmatrix}.
\]
 be two regular splittings of $A=\begin{bmatrix}
    A_{11} & A_{12}\\
    A_{2,1} & A_{22}
\end{bmatrix}$, with $A_{11}, M_{11},  N_{11}\in \mathbb R^{k\times k}$  and  $A_{22}, M_{22},  N_{22}\in \mathbb R^{(n-k)\times (n-k)}$.

Then, $\hat{M}^{-1}\hat{N}$ and $M^{-1}N$ have the same eigenvalues (and, hence, the same spectral radius).
\end{lemma}
\begin{proof}
We shall prove that the polynomials
\[
p(x) = \det(M)\det(xI - M^{-1}N) = \det(xM-N) = \det \underbrace{
\begin{bmatrix}
    xM_{11}-N_{11} & A_{12}\\
    xA_{21} & xM_{22} -N_{22}
\end{bmatrix}}_{:=P(x)}
\]
and
\[
q(x) = \det(\hat{M})\det(xI - \hat{M}^{-1}\hat{N}) = \det(x\hat{M}-\hat{N}) =\det
\underbrace{\begin{bmatrix}
    xM_{11}-N_{11} & xA_{12}\\
    A_{21} & xM_{22} -N_{22}
\end{bmatrix}}_{:=Q(x)}
\]
coincide, hence they have the same zeros.
For any $x\neq 0$ we have
\[
P(x)=\diag(I_k, xI_{n-k}) \cdot Q(x) \cdot \diag(I_k, x^{-1} I_{n-k}).
\]
The proof follows from the continuity  of the determinant w.r.t. the matrix entries. 
%By Laplace's formula, 
%\[
%p(x) = \sum_{\sigma \in \mathfrak{S}_n} \prod_{i=1}^n P_{i \sigma_i}(x), \quad q(x) = \sum_{\sigma \in \mathfrak{S}_n} \prod_{i=1}^n Q_{i \sigma_i}(x).
%\]
%For any fixed permutation $\sigma$, the two terms $\prod_{i=1}^n P_{i \sigma_i}(x)$ and $Q_{i \sigma_i}(x)$ coincide apart from a power of $x$; more precisely, one contains a factor $x^{|\mathcal{P}|}$, where $\mathcal{P} = \%{\text{$(i, \sigma_i)$ is in the block $(2,1)$}\}$, while the other contains a factor $x^{|\mathcal{Q}|}$, where $\mathcal{Q} = \{\text{$(i, \sigma_i)$ is in the block $(1,2)$}\}$. But $|\mathcal{P}| = |\mathcal{Q}|$, thanks %to a combinatorial argument: if there are $h$ elements in $1:k$ whose image is in $k+1:n$, then there must also be $h$ elements in $k+1:n$ whose image is in $1:k$.
\end{proof}

\begin{corollary} \label{cor:exchange}
Let
\[
M' = \begin{bmatrix}
    M_{11} & 0\\
    M_{21} & M_{22}
\end{bmatrix}, \quad
N' = \begin{bmatrix}
    N_{11} & -A_{12}\\
    N_{21} & N_{22}
\end{bmatrix}.
\]
and
\begin{equation} \label{hatsplitting}
    \hat{M} = \begin{bmatrix}
        M_{11} & A_{12}\\
        0 & M_{22}
    \end{bmatrix}, \quad
    \hat{N} = \begin{bmatrix}
        N_{11} & 0\\
        -A_{21} & N_{22}
    \end{bmatrix}.
\end{equation}
(where the blocks on the diagonal are square) be two regular splittings of $A$.
Then, $\rho(\hat{M}^{-1}\hat{N}) \leq \rho((M')^{-1}N')$.
\end{corollary}
\begin{proof}
By Lemma~\ref{lem:exchange} and Theorem~\ref{classicalcomparison}, $ \rho(\hat{M}^{-1}\hat{N}) = \rho(M^{-1}N) \leq \rho((M')^{-1}N')$.
\end{proof}

\section{Comparing the GS, AGS, and staircase splitting on Hessenberg matrices}\label{sect_compare}
Corollary~\ref{cor:exchange} shows that a regular splitting with $M(1\col k,k+1\col n)=0$ can be converted into one with $M(k+1\col n,1\col k)=0$ and $N(1\col k,k+1\col n)=0$  by decreasing
its spectral radius, that is,
equivalently by improving its  asymptotic rate of convergence. On a lower Hessenberg matrix, we can apply the lemma repeatedly for different values of $k$,
since each superdiagonal block $M(1\col k,k+1\col n)$ contains only one nonzero that does not overlap with blocks with a different $k$. In this way we obtain comparison results for different regular splittings. 

To be  more specific, let
$A\in \mathbb R^{n\times n}$ be a lower Hessenberg invertible M-matrix. Then
\begin{align*}
M_{J}&=\diag(A), &  N_{J}&=M_{J}-A, & P_J &= M_J^{-1}N_J\\ M_{GS}&=\tril(A), &  N_{GS}&=M_{GS}-A, & P_{GS} &= M_{GS}^{-1}N_{GS}
\end{align*}
are the customary Jacobi and Gauss--Seidel regular splittings. An easy modification of the Gauss--Seidel partitioning is the so called anti-Gauss--Seidel regular  splitting
defined by
\[
M_{AGS}=\triu(A), \quad  N_{AGS}=M_{AGS}-A, \quad P_{AGS} = M_{AGS}^{-1}N_{AGS}.
\]
Alternative regular splittings are analyzed in  the works  \cite{Meu,S1} which introduce the concept of stair partitioning of a matrix  aimed at  the design of
fast parallel (preconditioned)  iterative solvers.
Let $\tridiag(A)=\tridiag(A_{i,i-1},A_{i,i}, A_{i,i+1})$  be  the tridiagonal matrix formed from the subdiagonal, diagonal and superdiagonal entries of $A$.
The {\em stair matrix}  of first order generated by  $A$ is  the  $n\times n$  matrix   $S_1=\mathcal S_1(A)$  filled with the entries  of $A$ according to the following rule:
\begin{lstlisting}[mathescape]
$S_1=\tridiag(A)$;   for   $i=1:2:n$;  $S_1(i,i-1)=0$; $S_1(i,i+1)=0$.
\end{lstlisting}
Analogously, the stair matrix of second  order generated by  $A$ is  the  $n\times n$  matrix   $S_2=\mathcal S_2(A)$ defined by
\begin{lstlisting}[mathescape]
$S_1=\tridiag(A)$;   for   $i=2:2:n$;  $S_1(i,i-1)=0$; $S_1(i,i+1)=0$.
\end{lstlisting}
If  $S=\mathcal S(A)$ is the stair matrix of the first or second order constructed from $A$ then
\[
M_{S}=S, \quad N_{S}=M_{S}-A, \quad P_S = M_S^{-1}N_S
\]
gives a  staircase  regular splitting of $A$. The next result compares the asymptotic convergence rates of these splittings.  Recall that
the  classical inequality
\[
\rho(P_{J})\geq \rho(P_{GS})
\]
easily follows from Theorem \ref{classicalcomparison}.

\begin{theorem}\label{main}
  Let $A$ be a lower Hessenberg invertible M-matrix. Then,
  \[
  \rho(P_{GS}) \geq \rho(P_{S}) \geq \rho(P_{AGS}).
  \]
\end{theorem}
\begin{proof}
  The proof consists in applying Corollary~\ref{cor:exchange} repeatedly for all odd values of $k$, and then for all even values of $k$.
  We depict the transformations in the case $n=7$.
  We show here the nonzero pattern of $M$, displaying in position $i,j$ a symbol $\times$ to
  denote a nonzero entry $M_{ij}=A_{ij}$, and an empty cell to denote a zero entry $M_{ij}=0$.
  The symbol $\mapsto$ is used to denote a transformation of $M$ that reduces the spectral radius by Corollary~\ref{cor:exchange}.
\begin{align*}
M_{GS} &= 
\left[
\begin{array}{c|cccccc}
\times \\ \hline
\times & \times \\
\times & \times & \times \\
\times & \times & \times & \times \\
\times & \times & \times & \times & \times \\
\times & \times & \times & \times & \times & \times \\
\times & \times & \times & \times & \times & \times & \times \\
\end{array}
\right]
\mapsto
\left[
\begin{array}{ccc|cccc}
\times & \times \\
 & \times \\
 & \times & \times \\ \hline
 & \times & \times & \times \\
 & \times & \times & \times & \times \\
 & \times & \times & \times & \times & \times \\
 & \times & \times & \times & \times & \times & \times \\
\end{array}
\right]
\\
& \mapsto
\left[
\begin{array}{ccccc|cc}
\times & \times \\
 & \times \\
 & \times & \times & \times \\ 
 &  &  & \times \\
 &  &  & \times & \times \\ \hline
 &  &  & \times & \times & \times \\
 &  &  & \times & \times & \times & \times \\
\end{array}
\right]
\mapsto
\left[
\begin{array}{ccccccc}
\times & \times \\
 & \times \\
 & \times & \times & \times \\ 
 &  &  & \times \\
 &  &  & \times & \times & \times \\
 &  &  &  &  & \times \\
 &  &  &  &  & \times & \times \\
\end{array}
\right] = M_{S}.
\end{align*}
This sequence of transformations shows that $\rho(P_{GS}) \geq \rho(P_{S})$. We then continue with even values of $k$ to obtain the anti--Gauss--Seidel splitting.
\begin{align*}
M_{S} &=
\left[
\begin{array}{cc|ccccc}
\times & \times \\
 & \times \\ \hline
 & \times & \times & \times \\ 
 &  &  & \times \\
 &  &  & \times & \times & \times \\
 &  &  &  &  & \times \\
 &  &  &  &  & \times & \times \\
\end{array}
\right]
\mapsto
\left[
\begin{array}{cccc|ccc}
\times & \times \\
 & \times & \times \\ 
 &  & \times & \times \\ 
 &  &  & \times \\ \hline
 &  &  & \times & \times & \times \\
 &  &  &  &  & \times \\
 &  &  &  &  & \times & \times \\
\end{array}
\right]
\\ & \mapsto
\left[
\begin{array}{cccccc|c}
\times & \times \\
 & \times & \times \\ 
 &  & \times & \times \\ 
 &  &  & \times &\times \\ 
 &  &  &  & \times & \times \\
 &  &  &  &  & \times \\ \hline
 &  &  &  &  & \times & \times \\
\end{array}
\right]
\mapsto
\left[
\begin{array}{ccccccc}
\times & \times \\
 & \times & \times \\ 
 &  & \times & \times \\ 
 &  &  & \times &\times \\ 
 &  &  &  & \times & \times \\
 &  &  &  &  & \times & \times \\ 
 &  &  &  &  & & \times \\
\end{array}
\right] = M_{AGS}.
\end{align*}
\end{proof}

\begin{remark}
  Theorem 4 extends straightforwardly to nonsingular M-matrices in
block Hessenberg form by considering the corresponding block regular splittings.
  \end{remark}

\section{Comparing substitution splittings for Hessenberg matrices} \label{sec:substitution}

\emph{Substitution splittings} provide a generalization of the partitionings introduced in the previous  section. 
A permutation $v = (i_1,i_2,\dots,i_n)$  of $(1,2,\dots,n)$ is said a \emph{substitution order} for $M\in \mathbb{R}^{n\times n}$ if $M_{ij}=0$ whenever $j$ comes after $i$ in the list $v$.
This property means that we can solve a system $Mx=b$ by substitution, computing unknowns $x_i$ in the order of the list $v$, as in the following pseudocode:
\begin{lstlisting}[mathescape]
for $i = i_1, i_2, \dots, i_n$  
  solve for $x_i$ in the $i$th row of $M\B x=\B b$;
end
\end{lstlisting}
We call \emph{substitution matrix} any matrix $M$ that admits a substitution order. This includes lower triangular matrices (with order $[1,2,\dots,n]$),  upper triangular matrices (with order $[n,n-1,\dots,1]$) and
staircase partitionings (with order $[1,3,\dots, 2,4, \ldots]$ or  $[2,4,\dots, 1,3, \ldots]$).  Substitution splittings  also comprise the generalized staircase partitionings  described in \cite{S1}.
More generally, it is not hard to see that $M$ is a substitution matrix if and only if $\Pi M \Pi^T$ is lower triangular, where $\Pi$ is the permutation matrix associated to $v$.

The following theorem generalizes the previous result on comparing splittings for Hessenberg matrices.
\begin{theorem}\label{maingen}
  Let $A$ be a lower Hessenberg invertible M-matrix, and let $(M,N)$ be a regular splitting with a substitution matrix $M$. Then,
  \[
  \rho(M^{-1}N) \geq \rho(M_{AGS}^{-1}N_{AGS}).
  \]
\end{theorem}
\begin{proof}
First of all, note that it is sufficient to consider splittings of the form
\begin{equation} \label{specialsplitting}
M_{ij} = \begin{cases}
0 & \text{$j$ comes after $i$ in $v$},\\
A_{ij} & \text{otherwise}.
\end{cases}
\end{equation}
for some permutation $v$. Indeed, if $M'$ is another substitution matrix with the same $v$, then $N' \geq N$ and hence $\rho((M')^{-1}N') \geq \rho(M^{-1}N)$ by Theorem~\ref{classicalcomparison}.

Take such a splitting $(M,N)$, and suppose that $M_{k,k+1} \neq A_{k,k+1}$ for some $k \in \{1,2,\dots,n\}$. Then, $M_{k,k+1}=0$, and we can apply Corollary~\ref{cor:exchange} with the first block of size $k\times k$ and show that the splitting $(\hat{M},\hat{N})$ as defined in in~\eqref{hatsplitting}
has $\rho(\hat{M}^{-1}\hat{N}) \leq \rho(M^{-1}N)$.

We claim that $\hat{M}$ is a substitution matrix, too. Indeed, consider the permutation obtained by concatenating $v_2 = v \cap \{k+1, k+2,\dots,n\}$ and $v_1 = v \cap \{1,\dots,k\}$ (with this notation we mean that the entries in the ordered lists $v_1,v_2$ come in the same order as in $v$). Then, this is a substitution order for $\hat{M}$: indeed, if $i\in v_2$ and $j\in v_1$, then $\hat{M}_{ij}=0$ by construction; while if $i,j$ belong both to $v_1$ (resp. $v_2$), with $i$ coming before $j$, then $\hat{M}_{ij}=M_{ij}=0$.

Hence we have obtained a new splitting with $\rho(\hat{M}^{-1}\hat{N}) \leq \rho(M^{-1}N)$ and a substitution matrix $\hat{M}$ that has one more superdiagonal nonzero element than $M$; we can repeat the process until we obtain $M_{AGS}$, which is the (unique) splitting of the form~\eqref{specialsplitting} with the maximal number of superdiagonal nonzero elements.
\end{proof}
\begin{remark}
  Theorem \ref{maingen}    also extends easily to nonsingular  M-matrices $A=(A_{ij})\in \mathbb R^{N\times N}$,  $A_{i,i}\in \mathbb R^{n_i\times n_i}$, $\sum_{i=1}^nn_i=N$,
  in block Hessenberg form  whenever we consider block regular splittings determined by substitution orders  $v = (i_1,i_2,\dots,i_n)$ acting on the block  entries. Adaptations of
   Theorem \ref{main} and  Theorem \ref{maingen} for upper and block upper  invertible M-matrices are  immediate.
  \end{remark}

\section{Singular Systems} \label{sec:singular}
The solution of homogeneous singular systems $A \B x=\B 0$  where $A$ is a singular M-matrix in (block) Hessenberg form  is of paramount importance  for application in Markov  chains.
Linear stationary iterative methods of the type \eqref{iterativeformula}  have been successfully used for this problem. 
The rate of convergence of these iterative methods is governed by the quantity $\gamma(P)=\max\{|\lambda| \colon \lambda \in \sigma(P), \lambda\neq 1\}$
 where  $\sigma(P)$  is the spectrum of $P$.  Under convergence conditions  this  quantity  is called   the asymptotic
 convergence factor of the iterative method  \eqref{iterativeformula} applied for the solution of homogeneous singular system  $A \B x=\B 0$.

 For the sake of clarity,  let  $A$  be  an irreducible  singular M-matrix in lower Hessenberg form. Recall that from  Theorem 4.16 in \cite{Plemmons} $A$ has  rank $n-1$ and, hence,
 $\dim\ker(A)=1$ (see also \cite{KLS} for a brief survey of properties of singular irreducible M-matrices).
 We assume, up to scaling, that $A = I - T$, where $T$ is a  column stochastic matrix (in particular, $T_{ij} \geq 0$ for all $i,j$). It holds
 $\B e^T A =\B 0^T$  and, hence,
 \[
B= L A=\begin{bmatrix}
    A[1\colon n-1, 1\colon n-1] & A[1\colon n-1, n]\\
   \B 0^T & 0
\end{bmatrix}, \quad L=\begin{bmatrix}
    I_{n-1} & \B 0\\
    \B e^T & 1
 \end{bmatrix}
 \]
 As $A$ is irreducible, it follows that $A_{n-1} := A[1\colon n-1, 1\colon n-1]$ is a nonsingular lower Hessenberg M-matrix.   This makes possible to construct
 Jacobi-like, Gauss--Seidel-like and staircase-like  regular partitionings of $A$ starting from their analogue for the matrix $B$. Specifically,
 let  us denote
 \[
 M'_{J}=L^{-1} \begin{bmatrix}
     \diag(A_{n-1}) & 0\\
     0^T & 1
 \end{bmatrix}, 
 \quad  M'_{GS}=L^{-1} \begin{bmatrix}
     \tril(A_{n-1}) & 0\\ 0^T & 1
 \end{bmatrix},
 \] 
 and
 \[
M'_{AGS}=L^{-1} \begin{bmatrix}
    \triu(A_{n-1}) & 0\\
    0^T & 1
\end{bmatrix}, 
\quad M'_{S}=L^{-1} \begin{bmatrix}
    \mathcal S(A_{n-1}) & 0\\ 
    0^T & 1
\end{bmatrix}.
\]

Theorem \ref{main}  implies the following.

\begin{theorem}\label{mains}
  Let $A$ be an irreducible  singular  lower Hessenberg M-matrix. Then,
  \[
 \gamma({M'_{J}}^{-1}N'_{J})\geq  \gamma({M'_{GS}}^{-1}N'_{GS}) \geq \gamma({M'_{S}}^{-1}N'_{S}) \geq \gamma({M'_{AGS}}^{-1}N'_{AGS}).
  \]
\end{theorem}
 
\section{A combinatorial argument for $\rho(P_{GS}) \geq \rho(P_{AGS})$}

M-matrices and non-negative matrices are intimately related to Markov chains and transition probabilities; hence the reader may wonder if the results presented here admit an alternative combinatorial proof based on comparing probabilities of certain walks on a Markov chain. We present briefly such a proof for the inequality $\rho(P_{GS}) \geq \rho(P_{AGS})$, to highlight the ideas behind the argument.

Up to scaling, we may assume $A = I - T$, where $T\in \mathbb{R}^{n\times n}$ is a \emph{substochastic matrix}, i.e., $T \geq 0$ and $T\B 1 \leq \B 1$, where $\B 1 \in \mathbb{R}^{n}$ is the vector of all ones. A \emph{walk} on the graph with vertices $\{1,2,\dots,n\}$ (from $i_0$ to $i_\ell$ of length $\ell$) is a sequence of consecutive edges (\emph{transitions}) $\omega = ((i_0,i_1), (i_1,i_2), \dots, (i_{\ell-1},i_\ell))$. To a set of walks $\Omega$ we associate the transition probability matrix
\[
\mathbb{P}[\Omega] = (\mathbb{P}[\Omega]_{ij}), \quad \mathbb{P}[\Omega]_{ij} = \sum_{\omega \in \Omega \text{ with } i_0=i, i_\ell=j} T_{i i_1} T_{i_1 i_2} \dots T_{i_{\ell-1}j}.
\]
The matrix entry $\mathbb{P}[\Omega]_{ij}$ can be interpreted as the probability of observing a walk from $i$ to $j$ belonging to $\Omega$ (conditioned on starting from $i_0=i$) in an absorbing Markov chain~\cite{GrinSnell} with transition matrix
\[
\begin{bmatrix}
    T & \B e\\
    0 & 1
\end{bmatrix} \in \mathbb{R}^{(n+1)\times (n+1)}, 
\quad \B e = (I-T)\B 1.
\]
We have added an additional absorbing state $n+1$ to account for the missing probability due to $T\B 1 \leq \B 1$ not being an equality. Note that $\mathbb{P}[\Omega]$ defined in this way is a substochastic matrix only when the set $\Omega$ is prefix-free; otherwise it may contain entries larger than 1 due to repeating prefixes: e.g., when $\Omega$ contains both $((1,2))$ and $((1,2), (2,3), (3,2))$, then $\mathbb{P}[\Omega]_{12}$ contains a sum of probabilities of events that are not disjoint.

We write $T = D + L + U$, where $L = \operatorname{tril}(T)$ is associated with \emph{downward} transitions, i.e., transitions from a level $i$ to a level $j<i$, and symmetrically $U = \operatorname{triu}(T)$ is associated with \emph{upward} transitions, i.e., transitions from a level $i$ to a level $j>i$.  The key combinatorial lemma is the following. 
\begin{lemma} \label{combolemma}
Let $A = I - T \in \mathbb{R}^{n\times n}$ (with $T$ substochastic) be lower Hessenberg, and $\omega$ be a walk with non-zero probability. If $\omega$ contains $k$ downward transitions, then it contains at least $k-n+1$ upward transitions.
\end{lemma}
\begin{proof}
Let us choose an integer $h$ such that $1\leq h < n$, and define
\begin{align*}
s_h^{\uparrow}(\omega) &= \text{no. of transitions in $\omega$ from a state in $\{1,2,\dots,h\}$ to one in $\{h+1,h+2,\dots,n\}$},\\
s_h^{\downarrow}(\omega) &= \text{no. of transitions in $\omega$ from a state in $\{h+1,h+2,\dots,n\}$ to one in $\{1,2,\dots,h\}$}.
\end{align*}
Clearly, once we reach $\{h+1,h+2,\dots,n\}$ we must leave it before entering it again, hence transitions of the two kinds alternate in $\omega$, and thus 
\[
s_h^{\uparrow}(\omega) -1 \leq s_h^{\downarrow}(\omega) \leq s_h^{\uparrow}(\omega) + 1.
\]
We take the rightmost inequality and sum over $h$, to get
\[
s^{\downarrow}(\omega) \leq \sum_{h=1}^{n-1} s_h^{\downarrow}(\omega) \leq \sum_{h=1}^{n-1} (s_h^{\uparrow}(\omega) + 1) = s^{\uparrow}(\omega) + (n-1).
\]
Here, $s^{\uparrow}(\omega)$ is the total number of upward transitions in $\omega$, which is equal to $\sum_{h=1}^{n-1} s_h^{\uparrow}(\omega)$ because in a lower Hessenberg matrix each upward transition with nonzero probability is of the form $(h,h+1)$ for some $h$. On the other hand, $s^{\downarrow}(\omega)$ is the number of downward transitions in $\omega$, which is smaller or equal than the sum $\sum_{h=1}^{n-1} s_h^{\downarrow}(\omega)$, since each downward transition is counted in $s_h^{\downarrow}(\omega)$ for at least one choice of $h$, but may be counted in multiple ones: for instance, a transition from state $4$ down to state $1$ is counted in $s^\downarrow_1(\omega), s^\downarrow_2(\omega)$, and $s^\downarrow_3(\omega)$.
\end{proof}
From the lemma we can obtain an alternative proof of the following result.
\begin{theorem}
Let $A = I - T \in \mathbb{R}^{n\times n}$ (with $T$ substochastic) be lower Hessenberg. Then, $\rho(P_{GS}) \geq \rho(P_{AGS})$.
\end{theorem}
\begin{proof}
By standard arguments for sojourn and hitting probabilities in Markov chains, we have
\begin{multline*}
\mathbb{P}[\text{walks of any length $\ell$ with $0$ downward transitions}] \\
= I + (D+U) + (D+U)^2 + (D+U)^3 + \dots = (I-D-U)^{-1},
\end{multline*}
and
\begin{multline*}
\mathbb{P}[\text{walks with exactly $k$ downward transitions, ending with a downward transition}] \\
= (I-D-U)^{-1}L(I-D-U)^{-1}L(I-D-U)^{-1}L\dotsm (I-D-U)^{-1}L = P_{AGS}^k,
\end{multline*}
where $P_{AGS} = (I-D-U)^{-1}L$ is the iteration matrix of the anti-Gauss--Seidel method. Similarly, when $k\geq n-1$, we have
\begin{multline*} \label{doubleprob} 
\mathbb{P}[\text{walks with \emph{at least} $k-n+1$ upwards transitions}] \\
= (I-D-L)^{-1}U(I-D-L)^{-1}U(I-D-L)^{-1}U\dotsm (I-D-L)^{-1}U(I+T+T^2+\dots)\\
 = P_{GS}^{k-n+1}(I-T)^{-1},
\end{multline*}
since a walk with at least $k-n+1$ upward transitions can be seen as a walk with exactly $k-n+1$ upward transitions, ending with one of them, followed by any walk. In view of Lemma~\ref{combolemma}, the following inequality hold component-wise
\begin{equation} \label{compineq}
    P^k_{AGS} \leq P_{GS}^{k-n+1}(I-T)^{-1}, \quad \text{for all $k\geq n-1$}, 
\end{equation}
and from~\eqref{compineq} we get
\[
\norm{P^k_{AGS}}_{\infty}^{1/k} \leq \norm{P_{GS}^{k-n+1}}_{\infty} ^{1/k} \,\, \norm{(I-T)^{-1}}_\infty ^{1/k}.
\]
Passing to the limit and using Gelfand's formula $\lim_{k\to\infty}\norm{M^k}^{1/k} = \rho(M)$, we obtain $\rho(P_{AGS}) \leq \rho(P_{GS})$.
\end{proof}

\begin{remark}
Equation~\eqref{compineq} shows that this comparison theorem follows from an elementwise inequality, although a more complicated one than the ones considered in~\cite{CsoV,Woz}.
\end{remark}

\section{Numerical experiments}

To verify the statements of the theorems and give a quantitative assessment of the difference between $\rho(P_{GS})$, $\rho(P_{S})$, $\rho(P_{AGS})$, we plot them for various examples.

To obtain a random lower Hessenberg M-matrix $A\in\mathbb{R}^{n\times n}$, we generate random non-negative $P\in\mathbb{R}^{n\times n}$, $\B u,\B v\in\mathbb{R}^n$ by drawing their entries uniformly and independently from $[0,1]$, and then we find the unique lower Hessenberg matrix $A$ such that 
\[
A\B u = \B v \quad \text{ and }\quad  A_{ij}=-P_{ij} \quad \text{for all $i\neq j$, $i \geq j-1$.}
\]
In Matlab code:
\begin{lstlisting}
P = tril(rand(n), 1);
P = P - diag(diag(P));
u = rand(n, 1);
v = rand(n, 1);
d = (v + P*u) ./ u;
A = diag(d) - P;
\end{lstlisting}
The condition $A\B u = \B v$ ensures that $A$ is a nonsingular M-matrix, since any Z-matrix for which $A\B u = \B v$ for some $\B u>0, \B v\geq 0, \B v\neq 0$ is an M-matrix~\cite[Chapter~6, Condition I\textsubscript{28}]{Plemmons}.

We show in Figure~\ref{fig:exp_compare} the values of the three iteration radii for 50 random $5\times 5$ lower Hessenberg matrices. One can confirm that the three values are always in the order predicted by Theorem~\ref{main}; moreover, the experiments reveal that while the difference between $\rho(P_{GS})$ and $\rho(P_{S})$ if often minimal, the difference with $\rho(P_{AGS})$ is much more substantial.
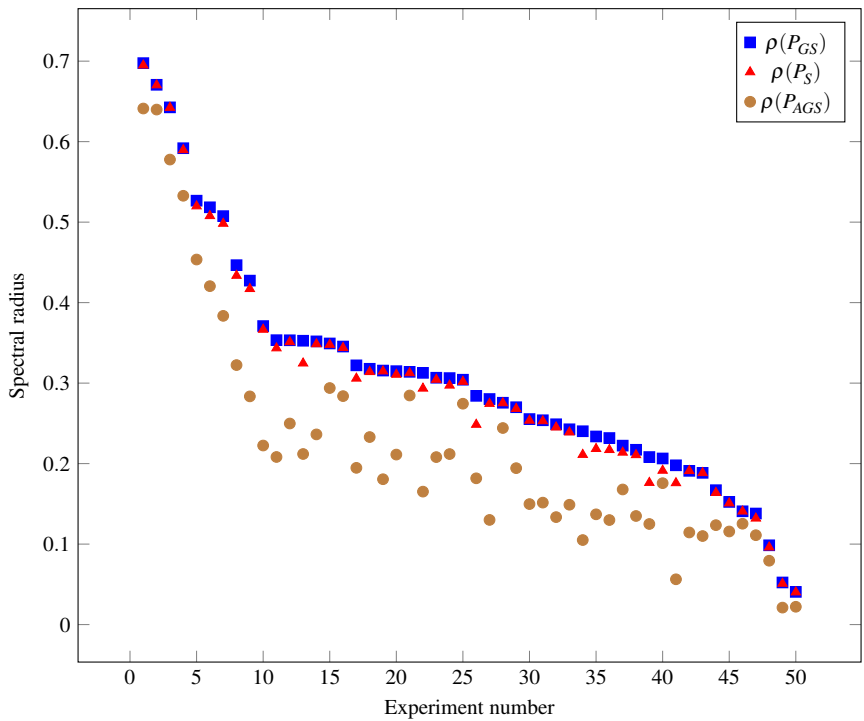
\begin{figure}[htp]
\begin{tikzpicture}
\begin{axis}[width=\textwidth, xlabel={Experiment number}, ylabel={Spectral radius}]
\pgfplotstableread{exp1.csv}{\mytable};
\addplot[only marks, mark=square*, blue] table[x=expnumber, y=rhoGS] {\mytable};
\addplot[only marks, mark=triangle*, red] table[x=expnumber, y=rhoS] {\mytable};
\addplot[only marks, mark=*, brown] table[x=expnumber, y=rhoAGS] {\mytable};
\legend{$\rho(P_{GS})$, $\rho(P_{S})$, $\rho(P_{AGS})$}
\end{axis}
\end{tikzpicture}
\caption{Comparison of iteration radii for 50 random $5\times 5$ lower Hessenberg M-matrices, sorted by decreasing value of $\rho(GS)$.} \label{fig:exp_compare}
\end{figure}

% In Figure~\ref{fig:exp_dimension}, we investigate what happens as the size of the matrices change. There seems to be no clear dependence on the dimension (probably due also to our choice of random matrices for the experiments).
% \begin{figure}[htp]
% \begin{tikzpicture}
% \begin{axis}[width=\textwidth, xlabel={Matrix size $n$}, ylabel={Spectral radius}]
% \pgfplotstableread{expdimension.csv}{\mytable};
% \addplot+[only marks, mark=square*] table[x=dimension, y=rhoGS] {\mytable};
% \addplot+[only marks, mark=triangle*] table[x=dimension, y=rhoS] {\mytable};
% \addplot+[only marks] table[x=dimension, y=rhoAGS] {\mytable};
% \legend{$\rho(P_{GS})$, $\rho(P_{S})$, $\rho(P_{AGS})$}
% \end{axis}
% \end{tikzpicture}
% \caption{Comparison of iteration radii for random lower Hessenberg M-matrices of varying size.} \label{fig:exp_dimension}
% \end{figure}
% %%%% TODO: remove this newpage -- it is here just because otherwise Latex chokes on these plots; but probably things will change as we write more text.
% \newpage
%

In Figure~\ref{fig:exp_excess}, we investigate what happens as the matrices $A$ get close to singular. For this experiment, rather than taking random $\B u,\B v$, we set $\B u = \B 1$ and $\B v = \eta \B 1$ for a certain scalar $\eta$. The first plot displays the three spectral radii; one sees that as $\eta$ gets smaller they get closer to $1$ (i.e., the iterative methods get slower) and the difference between them is harder to detect. For this reason, in a second plot we display the value of $k$ needed to obtain $\rho(P)^k \leq 0.01$ (that is, $k = \frac{\log 0.01 }{\log \rho(P)}$). This quantity gives an indication of the number of iterations required for the convergence of an iterative method, and is a more practical metric for this case. One can see that the benefits of the anti-Gauss--Seidel splitting in terms of convergence speed are present even in cases where the matrix is closer to singular.
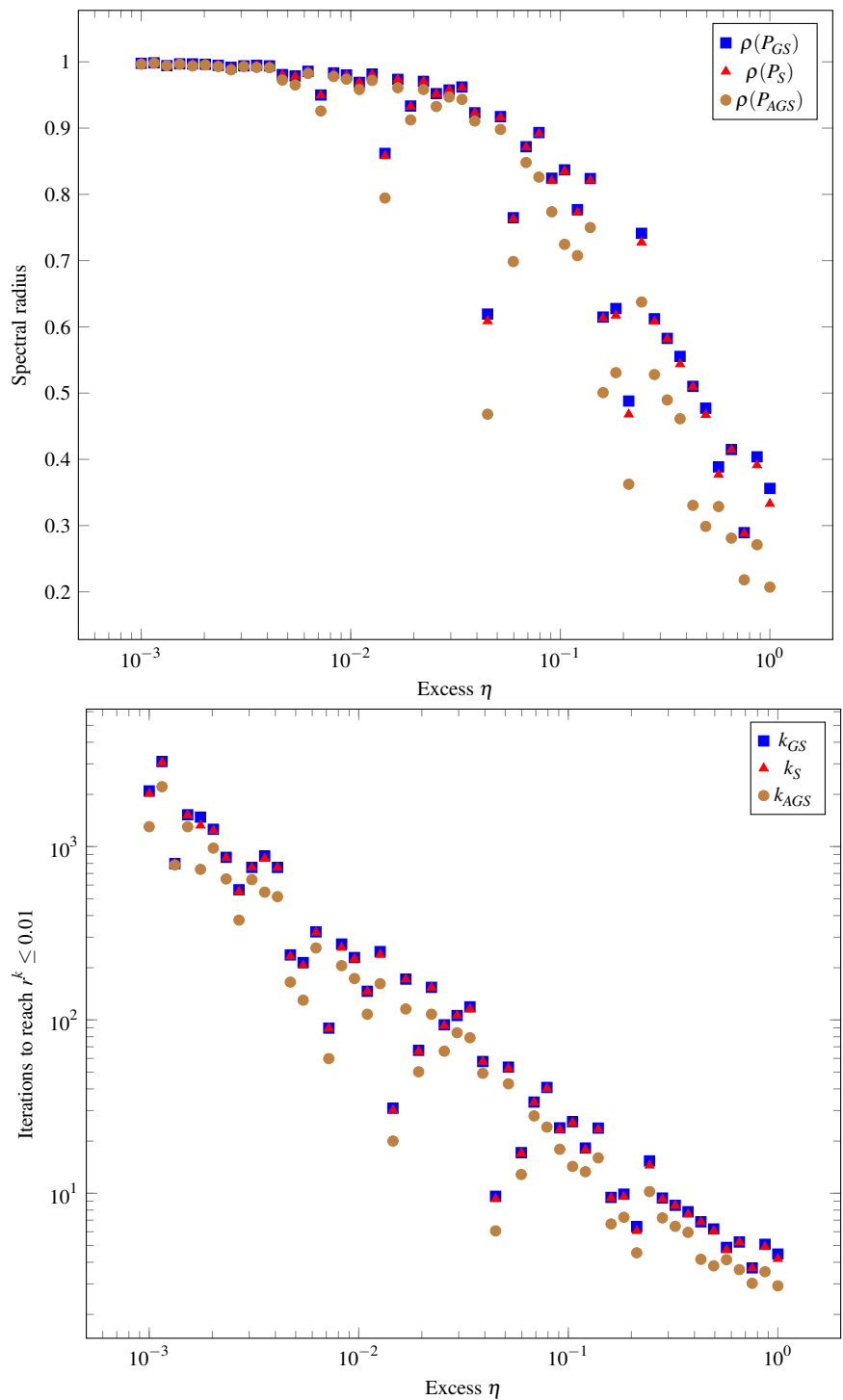
\begin{figure}[htp]
\begin{tikzpicture}
\begin{axis}[width=\textwidth, xlabel={Excess $\eta$}, ylabel={Spectral radius}, xmode=log]
\pgfplotstableread{expexcess.csv}{\mytable};
\addplot[only marks, mark=square*, blue] table[x=excess, y=rhoGS] {\mytable};
\addplot[only marks, mark=triangle*, red] table[x=excess, y=rhoS] {\mytable};
\addplot[only marks, mark=*, brown] table[x=excess, y=rhoAGS] {\mytable};
\legend{$\rho(P_{GS})$, $\rho(P_{S})$, $\rho(P_{AGS})$}
\end{axis}
\end{tikzpicture}

\begin{tikzpicture}
\begin{axis}[width=\textwidth, xlabel={Excess $\eta$}, ylabel={Iterations to reach $r^k \leq 0.01$}, ymode=log, xmode=log]
\pgfplotstableread{expexcess.csv}{\mytable};
\addplot[only marks, mark=square*, blue] table[x=excess, y=logGS] {\mytable};
\addplot[only marks, mark=triangle*, red] table[x=excess, y=logS] {\mytable};
\addplot[only marks, mark=*, brown] table[x=excess, y=logAGS] {\mytable};
\legend{$k_{GS}$, $k_{S}$, $k_{AGS}$}
\end{axis}
\end{tikzpicture}

\caption{Comparison of (a) iteration radii (b) power $k$ needed to reach $\rho(P)^k \leq 0.01$, for 50 random $5\times 5$ lower Hessenberg M-matrices with varying excess $\eta$.} \label{fig:exp_excess}
\end{figure}

For a more challenging example from applications, we take the generator matrix $Q$ from the queuing model described in~\cite{dudin}. This is a complex queuing model, a BMAP/PHF/1/N model with retrial system with finite buffer and non-persistent customers. We do not describe in detail the construction of this matrix, as it would take some space, but refer the reader to~\cite[Sections~4.3 and~4.5]{dudin}. The only change with respect to the paper is that we fix the orbit size to a finite capacity $K$ (when the orbit is full, customers leave the queue forever). We set $N=5$, which results in a block upper Hessenberg matrix $Q$ with $K\times K$ blocks of size $48$. This matrix $Q$ is singular $M$-matrix, as it is the generator of an irreducible continuous-time Markov chain; hence we can apply the strategy in Section~\ref{sec:singular} to $A=-Q$ to solve the system $\pi Q = 0$ and determine the invariant distribution of the queue. In addition, the matrix is quite sparse: for $K=50$ it has 1.1\% of nonzero elements. This sparsity makes the use of iterative methods appealing, especially if one has in mind simulations with large values of $N$ and $K$. In this paper, we restrict to small and medium-sized matrices, so that we can comfortably compute their spectral radii with the Matlab function \texttt{eig}.

In Figure~\ref{fig:exp_dudin}, we plot once again the number of iterations required to reach $r^k \leq 0.01$, for various values of $K$, using the block variants of the iterative methods described above.
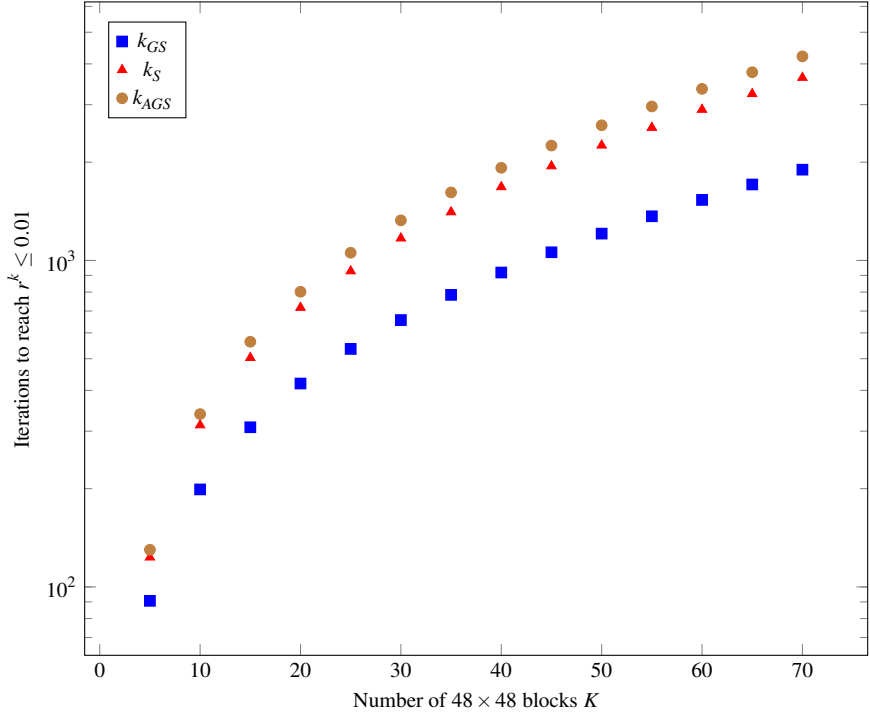
\begin{figure}[htp]
\begin{tikzpicture}
\begin{axis}[width=\textwidth, xlabel={Number of $48\times 48$ blocks $K$}, ylabel={Iterations to reach $r^k \leq 0.01$}, ymode=log, legend pos = north west]
\pgfplotstableread{experiment_dudin.csv}{\mytable};
\addplot[only marks, mark=square*, blue] table[x=K, y=kGS] {\mytable};
\addplot[only marks, mark=triangle*, red] table[x=K, y=kS] {\mytable};
\addplot[only marks, mark=*, brown] table[x=K, y=kAGS] {\mytable};
\legend{$k_{GS}$, $k_{S}$, $k_{AGS}$}
\end{axis}
\end{tikzpicture}

\caption{Comparison of power $k$ needed to reach $\rho(P)^k \leq 0.01$, for various values of $K$ in the matrix $Q$ from~\cite{dudin}.} \label{fig:exp_dudin}
\end{figure}
Note that this time the matrix is upper Hessenberg, not lower, hence the results are reversed and the best method is Gauss-Seidel, with a block lower bidiagonal $M_{GS}$, as predicted by the theory.

We conclude the section with a numerical exploration of overrelaxation in the context of these iterative methods. As is well known, the main difficulty in using SOR is the choice of the parameter $\omega$. Our theoretical results do not cover the case $\omega > 1$, since it does not lead to regular splittings; however, it is still interesting to investigate the behavior of the iteration radius in this regime.

We recall briefly that GS and AGS with overrelaxation are given by
\[
    x^{(k+1)}_i = (1-\omega)x^{(k)}_i + \omega A_{ii}^{-1}\left( b - \sum_{j< i} A_{ij}x_j^{(k+1)} - \sum_{j > i} A_{ij}x_j^{(k)}\right), \quad i=1,2,\dots,n
\]
and
\[
    x^{(k+1)}_i = (1-\omega)x^{(k)}_i + \omega A_{ii}^{-1}\left( b - \sum_{j> i} A_{ij}x_j^{(k+1)} - \sum_{j < i} A_{ij}x_j^{(k)}\right), \quad i=1,2,\dots,n
\]
One can verify that these formulas correspond to taking $M_{GSOR,\omega} = \frac{1}{\omega} D + L$, $M_{AGSOR,\omega}= \frac{1}{\omega} D + U$, where $D,L,U$ stand for the (block) diagonal, strictly (block) lower triangular, and strictly (block) upper triangular part of the matrix $A$.

As for the staircase splitting, a natural choice to perform overrelaxation is taking
\[
    x^{(k+1)}_i = (1-\omega)x^{(k)}_i + \omega A_{ii}^{-1}\left( b - \sum_{j\in J_i} A_{ij}x_j^{(k+1)} - \sum_{j \not \in J_i, j\neq i} A_{ij}x_j^{(k)}\right), \quad i=1,2,\dots,n,
\]
where
\[
J_i = \begin{cases}
    \{i-1,i+1\} \cap \{1,2,\dots,n\} & \text{$i$ odd}, \\
    \emptyset & \text{$i$ even},
\end{cases}
\]
i.e., for each index $i$ one takes a weighted combination of the old entry $x_i^{(k)}$ and what would be the new entry computed with the staircase splitting. This yields a splitting with $M_{STSOR,\omega} = M_S + \frac{1-\omega}{\omega} D$, a matrix which coincides with $M_S$ on the off-diagonal entries (or blocks) but has the diagonal entries (blocks) scaled by a factor $\frac{1}{\omega}$.

This is not the only possible choice to define an over-relaxed version of the staircase splitting; a possible alternative is moving the weighted combination inside the parentheses and writing
\[
    x^{(k+1)}_i = A_{ii}^{-1}\left( b - \sum_{j\in J_i} A_{ij}\left((1-\omega)x_j^{k} + \omega x_j^{(k+1)}\right) - \sum_{j \not \in J_i, j\neq i} A_{ij}x_j^{(k)}\right), \quad i=1,2,\dots,n
\]
instead. This yields a splitting with $M_{STSOR2,\omega} = \omega M_S + (1-\omega) D$, i.e., a matrix which coincides with $M_S$ in the diagonal entries (blocks) but has the off-diagonal entries (blocks) scaled by a factor $\omega$.

In Figure~\ref{fig:exp_SOR} we investigate the effect of overrelaxation on the matrix $Q$ from the previous example, setting $K=20$ as the number of blocks.
\begin{figure}[htp] 
\begin{tikzpicture}
\begin{axis}[width=\textwidth, xlabel={Overrelaxation parameter $\omega$}, ylabel={Spectral radius}, ymax=1.07, legend pos = north west]
\pgfplotstableread{experiment_SOR.csv}{\mytable};
\addplot[only marks, mark=square*, blue] table[x=omega, y=rhoGSOR] {\mytable};
\addplot[only marks, mark=triangle*, red] table[x=omega, y=rhoSTSOR] {\mytable};
\addplot[only marks, mark=diamond*, purple] table[x=omega, y=rhoSTSOR2] {\mytable};
\addplot[only marks, mark=*, brown] table[x=omega, y=rhoAGSOR] {\mytable};
\legend{$\rho(P_{GSOR,\omega})$, $\rho(P_{STSOR,\omega})$, $\rho(P_{STSOR2,\omega})$, $\rho(P_{AGSOR,\omega})$}
\end{axis}
\end{tikzpicture}

\caption{Comparison of iteration radii for various values of the overrelaxation parameter $\omega$ in the matrix $Q$ from~\cite{dudin} (with $K=20$).} \label{fig:exp_SOR}
\end{figure}
In this example, the iteration radius decreases until a certain threshold is reached, but after that it increases sharply and quickly surpasses $1$ (and then the iterative method stops converging). This threshold strongly depends on the method used, with $STSOR2$ being the one that is less sensitive to the value of $\omega$. One can see that even for $\omega>1$ the inequality between $\rho(P_{GSOR,\omega})$, $\rho(P_{STSOR,\omega})$, $\rho(P_{AGSOR,\omega})$ (but not $\rho(P_{STSOR2,\omega})$) is preserved in this example.

It is  worth pointing out that for the optimal choice of the relaxation parameter  $\omega=1.33$ in GSOR we find  $\rho(P_{GSOR,1.33})\simeq( \rho(P_{S}))^{6.335}$. 
    Since one iteration of the block staircase  method requires
    only two parallel steps (as opposed to $n-1$ for GSOR) it follows that even in absence of overrelaxation the staircase splitting method is still more efficient in a parallel computing environment.
    Comparison of the methods for the best choice of the corresponding relaxation parameters  also yields considerably better results.  For instance, we find
   $\rho(P_{GSOR,1.33})\simeq( \rho(P_{STSOR2,1.89}))^{2.31}$.

In our last experiment, we consider an example with a block tridiagonal matrix. We take 
    $
    Q=-(A\otimes I_n + I_n \otimes A + \lambda_1\diag(0_{n-1}, 1)\otimes R)$,
    where $A\in \mathbb R^{n\times n}$ is  tridiagonal,
    \[
    A=\left[\begin{array}{cccccccc}
        \lambda & -\mu\\
        -\lambda & \lambda+\mu & -2\mu \\& -\lambda & \lambda+2\mu & -3\mu \\ & & \ddots & \ddots & \ddots \\
        & & & -\lambda & \lambda+s\mu & -s\mu \\ & & & &\ddots & \ddots & \ddots \\
        & & & & &  -\lambda & \lambda+s\mu & -s\mu \\ & & & & & & -\lambda & s\mu
      \end{array}\right]
    \]
    and $R\in \mathbb R^{n\times n}$  is lower bidiagonal with  diagonal and subdiagonal entries equal to $1$ and $-1$, respectively, and $R(n,n)=0$. We set $n=21$, $s=5$, $\lambda=0.9$, $\mu=0.1$ and $\lambda_1=1$.
    The matrix $Q$ is a singular block tridiagonal $M$-matrix, and is the generator of nearly-separable continuous-time Markov chain associated to a 2-queue overflow network~\cite{MHKS}. Iterative methods are effective for the solution of nearly-separable  Markov chains  due to the knowledge  of a good starting point \cite{OLDP}.
      From the results in Section~\ref{sect_compare}  and Section~\ref{sec:singular} it follows that GS, AGS, and staircase splittings have the same iteration radius in absence of overrelaxation. In Figure~\ref{fig:exp_SOR_compare} one sees that this property is preserved for all values of $\omega$ for GSOR, STSOR, AGSOR (at least in this example), but not for STSOR2.
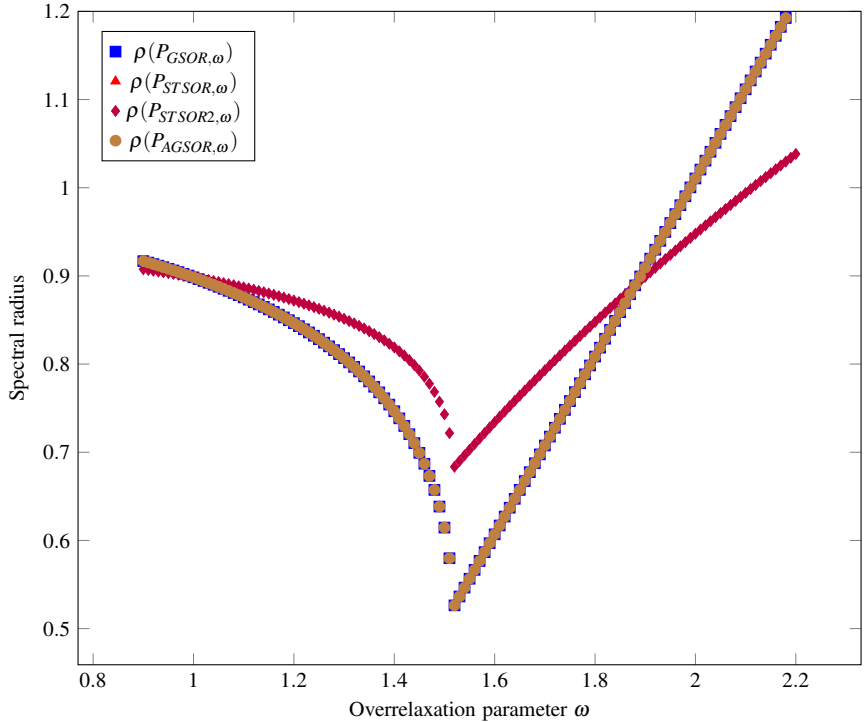
\begin{figure}[htp]
\begin{tikzpicture}
\begin{axis}[width=\textwidth, xlabel={Overrelaxation parameter $\omega$}, ylabel={Spectral radius}, ymax=1.2, legend pos = north west]
\pgfplotstableread{experiment_SOR_compare.csv}{\mytable};
\addplot[only marks, mark=square*, blue] table[x=omega, y=rhoGSOR] {\mytable};
\addplot[only marks, mark=triangle*, red] table[x=omega, y=rhoSTSOR] {\mytable};
\addplot[only marks, mark=diamond*, purple] table[x=omega, y=rhoSTSOR2] {\mytable};
\addplot[only marks, mark=*, brown] table[x=omega, y=rhoAGSOR] {\mytable};
\legend{$\rho(P_{GSOR,\omega})$, $\rho(P_{STSOR,\omega})$, $\rho(P_{STSOR2,\omega})$, $\rho(P_{AGSOR,\omega})$}
\end{axis}
\end{tikzpicture}

\caption{Comparison of iteration radii for various values of the overrelaxation parameter $\omega$ in the matrix $Q$ from  \cite{MHKS}. } \label{fig:exp_SOR_compare}
\end{figure}
Again,  by comparison of  the optimal choices for the relaxation parameter  we find  $\rho(P_{GSOR,1.52})\simeq( \rho(P_{STSOR2,1.52}))^{1.686}$.
Note that the block staircase partitioning is not block triangular and therefore the well known necessary condition $\omega \in (0,2)$ for the convergence of SOR (due to Kahan~\cite{kahan}) does not apply.

\section{Conclusions}

In this paper we have shown that there is a hierarchy among matrix splittings for M-matrices in Hessenberg form, covering the Gauss--Seidel, anti-Gauss--Seidel, and staircase partitionings, together with some generalizations. These results encourage further investigation into comparison theorems for M-matrix splittings, suggesting that this classical topic is far from being completely analyzed and solved. Future work is concerned with the analysis  of these generalizations for the design of efficient processor-oriented variants of staircase  splitting methods for parallel computation. Another interesting topic  is the comparison of  classical  stationary iterative methods and staircase splitting methods for semidefinite matrices, under suitable structures.

Finally, the analysis of relaxation techniques  applied to  (block) staircase  splittings  is  an ongoing research.

\bibliographystyle{spmpsci} 
\bibliography{msplit}
 
\end{document}